\documentclass[11pt]{amsart}
\usepackage[utf8]{inputenc}
\usepackage[english]{babel}
\usepackage{mathtools}
\usepackage{amsfonts}
\usepackage{hyperref}

\usepackage{caption}
\usepackage{subcaption}

\usepackage{tikz, tkz-euclide}
\usetikzlibrary{intersections}

\usepackage{amsthm}
\newtheorem{thm}{Theorem}
\newtheorem{coro}{Corollary}

\theoremstyle{remark}
\newtheorem{rmk}{Remark}

\DeclareMathOperator{\Vol}{Vol}
\DeclareMathOperator{\dist}{dist}

\newcommand{\del}{\partial}

\begin{document}
\title[Rellich-Christianson identities on polytopes]
      {Rellich-Christianson type identities for the
       Neumann data mass of Dirichlet eigenfunctions on polytopes}

\author{Antoine Métras}
\address{Département de mathématiques et de statistique, Université de 
         Montréal, 2920, Chemin de la Tour, Montréal, Qc, H3T 1J4, Canada}
\email{metrasa@dms.umontreal.ca}
\maketitle

\begin{abstract}
    We consider the Dirichlet eigenvalue problem on a simple polytope.
    We use the Rellich identity to obtain an explicit formula expressing
    the Dirichlet eigenvalue in terms of the Neumann data on the faces of the
    polytope of the corresponding eigenfunction. The formula is particular
    simple for polytopes admitting an inscribed ball tangent to all the
    faces. Our result could be viewed as a generalization of
    similar identities for simplices recently found by Christianson
    \cite{Chr17-1, Chr17-2}.
\end{abstract}

\section{Introduction and main results}

Let $P$ be a simple (i.e. non self-intersecting) $n$-dimensional polytope 
in $\mathbb{R}^n$ with faces $F_1, \dots, F_k$. Let $f$ be a 
solution of the Dirichlet problem with eigenvalue $\lambda$:
\begin{align*}
    \begin{cases}
        (\Delta + \lambda) f = 0 &\qquad \text{in } P \\
        f = 0 &\qquad \text{on } \del P
    \end{cases}
\end{align*}
where $\Delta$ is the Laplacian and $f$ is assumed
to be normalized $\|f\|_{L^2} = 1$.
We are interested in the Neumann data mass on each face of the
polytope, defined on the face $F_i$ by
\begin{align*}
    \int_{F_i} |\del_{\nu} f|^2 dV
\end{align*}
where $\del_{\nu}$ is the outward pointing normal derivative on $F_i$.

In the case of a compact manifold in $\mathbb{R}^n$, general upper and lower
bounds for the Neumann data on the whole boundary are known \cite{HaT02}, while
in the specific case of $n$-dimensional simplices, Christianson showed
the equidistribution of the Neumann data on the faces \cite{Chr17-2}.
Inspired by these results, we investigate, using the Rellich identity,
similar equalities for the Neumann data mass of a Dirichlet eigenfunction 
on the boundary of a  polytope.

As Christianson showed with the example of a square \cite{Chr17-1},
it is not true in general that the Neumann
data mass is equidistributed between each faces. So we will instead consider
the sum over all the faces.

We first define some notations to express our result in simple terms. 
Given a point $p$, we write $C_i(p)$ for the pyramid with base
the face $F_i$ and apex $p$. We define $\Vol_n^*(C_i(p))$
to be the signed volume of $C_i(p)$, with the sign given by the sign
of $-v \cdot \nu$ where $v$ is any vector from the base to the apex
of the pyramid and $\nu$ is the outward normal vector to $F_i$.
This is defined such that if the polytope is convex and $p$ lies inside it,
all the $\Vol_n^*(C_i(p))$ are positive, while some are
negative when $p$ is outside of the polytope (see figure \ref{fig-pyramids}
for an example).
We also let $\dist^*(p,F_i)$ be the signed distance between $p$
and $F_i$ with the same sign as $\Vol_n^*(C_i(p))$.
With this notation our result is

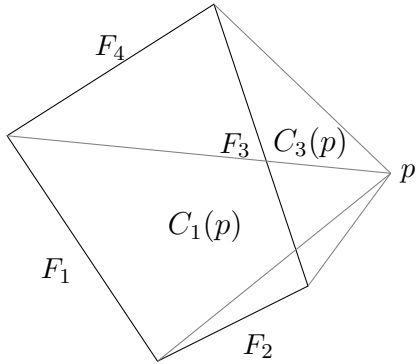
\begin{figure}
    \centering
    \begin{tikzpicture}
        \coordinate [label=right:$p$]   (p)     at (5.1, -0.5);
        \coordinate                     (p1)    at (0,0);
        \coordinate                     (p2)    at (2,-3);
        \coordinate                     (p3)    at (4,-2);
        \coordinate                     (p4)    at (2.75, 1.75);

        \draw   (p1) -- node[below left]    {$F_1$}
                (p2) -- node[below right]   {$F_2$}
                (p3) -- node[left]          {$F_3$}
                (p4) -- node[above]         {$F_4$}
                cycle;

        \foreach \p in {(p1), (p2), (p3), (p4)}
            \draw [gray] \p -- (p);
        \coordinate [label=right:{\large $C_1(p)$}] (c1) at (2,-1.2);
        \coordinate [label=right:{\large $C_3(p)$}] (c3) at (3.4, -0.1);
    \end{tikzpicture}
    \caption{For a quadrilateral, the pyramids 
    $C_1(p)$ and $C_3(p)$ are shown, with this specific
    choice of $p$ making $\Vol_2^*(C_1(p))$
    positive while $\Vol_2^*(C_3(p))$ is negative.}
    \label{fig-pyramids}
\end{figure}

\begin{thm} \label{main-thm}
    Let $P$ a simple $n$-polytope with faces $F_1, \dots, F_k$ and let 
    $p \in \mathbb{R}^n$ be an arbitrary point. Then
    \begin{align} \label{eq:main}
        \sum_{i = 1}^k \dist^*(p,F_i) \int_{F_i} |\del_\nu f|^2 dV =
        2 \lambda
    \end{align}
    or equivalently
    \begin{align}
        \sum_{i = 1}^k \frac{\Vol_n^*(C_i(p))}{\Vol_{n-1}(F_i)}
        \int_{F_i} |\del_{\nu} f|^2 dV = \frac{2\lambda}{n}
    \end{align}
\end{thm}

\begin{rmk} \label{rmk-eq-trig}
    If there exists a point $p$ making all the pyramids $C_1(p), \dots,
    C_k(p)$ have the same signed volume, the second equality can be rewritten as
    \begin{align}\label{eq:trig-equal}
        \sum_{i = 1}^k \frac{\int_{F_i} |\del_{\nu} f|^2 dV}{\Vol_{n-1}(F_i)}
        = \frac{2 k \lambda}{n \Vol_n(P)}.
    \end{align}
    Such a point does not exist for all polytopes. For example, on
    a polygon $P$, this point would be at the intersection of lines
    $L_1, \dots, L_k$ where $L_i$ is parallel to $F_i$ with a distance of
    $\frac{2\Vol_2(P)}{k\Vol_1(F_i)}$  between the two. It is clear
    when considering a trapezoid that those lines need
    not intersect at a common point.
\end{rmk}

\begin{rmk}
    Christianson's result \cite{Chr17-2} on simplices can be 
    obtained by putting $p$ on one of the vertex of the
    simplex to obtain the equality for the face opposite to it.
    Indeed in this case, all the pyramid $C_i(p)$ have volume 0
    except the one with base the opposite face. In fact the proof
    of the theorem was inspired by his proof in \cite{Chr17-1}
    and simplified using the Rellich identity.
\end{rmk}

The equality \eqref{eq:main} implies a simpler result when considering
polytopes with special geometric properties, 

\begin{coro}
    If $P$ is a tangential polytope, i.e. $P$ has an inscribed ball
    tangent to all its faces, then
    \begin{align}
        \int_{\del P} |\del_{\nu} f|^2 dV
        = \frac{\Vol_{n-1}(\del P)}{\Vol_n(P)} \frac{2 \lambda}{n} .
    \end{align}
\end{coro}

Restricting ourselves to polygons, one can also follow up on remark 
\ref{rmk-eq-trig} and try
to find geometric properties of polygon with inner triangles of same area:

\begin{coro}
    If $Q$ is a quadrilateral with one diagonal intersecting the
    other diagonal in its middle then there exists a point $p$
    such that the $\Vol_{2}^*(C_i(p))$ are all equal and
    \begin{align} \label{eq:quadri}
        \sum_{i = 1}^4 \frac{\int_{F_i} |\del_{\nu} f|^2 dV}{
        \Vol_{1}(F_i)} = \frac{4 \lambda}{\Vol_{2}(Q)} .
    \end{align}
\end{coro}

\begin{rmk}
    This corollary applies to any parallelogram.
\end{rmk}

\section{Proofs}

\begin{proof}[Proof of Theorem \ref{main-thm}]
Without loss of generality, after a translation
in $\mathbb{R}^n$, we can assume $p$ is at the origin.
Let $p_i$ be the orthogonal projection of $p$ on $F_i$.
In the following, we consider points in $\mathbb{R}^n$ as vectors.

We consider Rellich's identity for Dirichlet eigenvalues \cite{Rel40}:
\begin{align}
    2 \lambda = \int_{\del P} (\nu, q) |\del_{\nu} f|^2 dV(q).
\end{align}
This integral is split over the faces $F_i$ of the polytope, 
writing $\nu_i$ for the normal vector to $F_i$.
One then remark that $F_i \subset p_i + p_i^{\perp}$, where $p_i^{\perp}$
is the orthogonal hyperplane to $p_i$, and by construction, $p_i$ and 
$\nu_i$ are parallel so $(\nu_i, q) = (\nu_i, p_i)$ for all $q \in F_i$.
Thus
\begin{align*}
    \int_{F_i} (\nu_i, q) |\del_{\nu} f|^2 dV(q) &=
    (\nu_i, p_i) \int_{F_i} |\del_{\nu} f|^2 dV.
\end{align*}
Defining $I_i = \int_{F_i} |\del_{\nu} f|^2 dV$ and summing over all the
sides of the polytope, we obtain the first equality of the theorem
\begin{align*}
    2 \lambda &= \int_{\del P} (\nu, q) |\del_{\nu} f|^2 dS(q) 
    = \sum_{i = 1}^k (\nu_i, p_i) I_i \\
    &= \sum_{i = 1}^k |p_i| (\nu_i, p_i/|p_i|) I_i 
    = \sum_{i = 1}^k \dist^*(p, F_i) I_i
\end{align*}
Here $\dist^*(p, F_i) = |p_i| (\nu_i, p_i/|p_i|) = \dist(p, F_i)
(\nu, p_i/|p_i|)$ is the signed distance
between $p$ and $F_i$. Finally we use the formula for the volume
of a pyramid, $\Vol_n(C_i(p)) = \frac{1}{n} \dist(p, F_i) \Vol_{n-1}(F_i)$,
to obtain
\begin{align*}
    \frac{2 \lambda}{n} = \sum_{i = i}^k \frac{\Vol_n^*(C_i(p))}{\Vol_{n-1}
    (F_i)} I_i
\end{align*}

\end{proof}

\begin{proof}[Proof of Corollary 1.]
In the case of a tangential polytope $P$ the point $p$, apex of the
pyramids, is set to be the center of the inscribed ball. Then
all the distances $\dist^*(p, F_i)$ are equal to some $h$ by definition of
tangential polytope. Then by the first equality of theorem \ref{main-thm},
\begin{align*}
    \sum_{i = 1}^k \int_{F_i} |\del_\nu f|^2 dV = \frac{2 \lambda}{h}
\end{align*}

We also have for the volume of the boundary $\del P$,
\begin{align*}
    \Vol_{n-1}(\del P) &= \sum_{i = 1}^k \Vol_{n-1} (F_i)
    = \sum_{i = 1}^k \frac{n \Vol_n(C_i(p))}{h}
    = \frac{n}{h} \Vol_n(P)
\end{align*}
thus
\begin{align*}
    \sum_{i = 1}^k \int_{F_i} |\del_\nu f|^2 dV = \frac{\Vol_{n-1}(\del P)}
    {\Vol_n(P)} \frac{2 \lambda}{n}
\end{align*}

\end{proof}

\begin{proof}[Proof of Corollary 2.]
The vertices of the quadrilateral $Q$ are labeled $A,B,C,D$ and the
diagonal $D_1, D_2$ between $A,C$ and $B,D$ respectively. We suppose
without loss of generality that $D_1$ passes through the middle of
$D_2$ and name $E$ their intersection point (see figure \ref{fig-quad}).

\begin{figure}
    \centering
    \begin{tikzpicture}
        \coordinate [label=left:$A$]        (A)  at (0,0);
        \coordinate [label=below left:$B$]  (B)  at (0.5, -2);
        \coordinate [label=below right:$B'$](Bh) at (0.5,0);
        \coordinate [label=below right:$C$] (C)  at (4,0);
        \coordinate [label=above right:$D$] (D)  at (5.5, 2);
        \coordinate [label=right:$D'$]      (Dh) at (5.5, 0);
        \coordinate [label=above right:$p$] (p)  at (2,0);

        \node [below] (E) at ($ (B)!.5!(D) $) {$E$};

        \draw (A) -- (B) -- (C) -- (D) -- cycle;
        \draw (A) -- (C);
        \draw (B) -- (D);
        \tkzMarkSegment[pos=0.25,mark=||](B,D);
        \tkzMarkSegment[pos=0.75,mark=||](B,D);

        \draw [gray] (B) -- (p) -- (D);

        \draw [dashed] (B) -- (Bh);
        \draw [dashed] (D) -- (Dh) -- (C);
    \end{tikzpicture}
    \caption{}
    \label{fig-quad}
\end{figure}

We will show that by taking $p$ in the middle of $D_1$, the
triangles $ABp$, $BCp$, $CDp$ and $DAp$ have the same area thus
allowing us to use equation \eqref{eq:trig-equal}.

Let $B'$ and $D'$ be the orthogonal projection on $AC$ of $B$ and 
of $D$ respectively. The angles $\widehat{B'EB}$ and $\widehat{DED'}$
are equal and so are $\widehat{BB'E}$ and $\widehat{DD'E}$. 
Thus the triangles $B'BE$ and $DD'E$ are similar. But since
$E$ is the middle of $BD$, the edges $BE$ and $ED$ have the same length
which implies that the triangles $B'BE$ and $DD'E$ are in fact congruent.
Thus the lengths of $BB'$ and $DD'$ are equal. 

By taking $p$ in the middle of $AC$, we obtain that $Ap$ and $Cp$
have the same length. Then since the area of a triangle is given
by the length of its base times its height, we conclude that
the triangles $ABp$, $BCp$, $CDp$ and $DAp$ have the same area.
Equation \eqref{eq:trig-equal} then gives the desired result.

\end{proof}

\subsection*{Acknowledgements}
    This paper is part of  a MSc thesis written under the
    supervision of Iosif Polterovich. I would like to thank him
    for bringing to my attention the Rellich identity, which greatly
    simplified my proof.

\bibliographystyle{hunsrt}
\bibliography{biblio}

\end{document}